\documentclass[a4paper,11pt]{article}
\usepackage{braket,hyperref}
\usepackage{amssymb,mathtools,amsthm,amsmath,bm}
\usepackage{tikz}

\usepackage{xargs}
\usepackage{gensymb}

\usepackage{enumitem}
\usepackage{verbatim}

\usepackage{kbordermatrix}%

\renewcommand{\kbldelim}{(}
\renewcommand{\kbrdelim}{)}
\usepackage{thm-restate}

\theoremstyle{plain}
\newtheorem{theorem}{\bf Theorem}[section]

\newtheorem{conjecture}[theorem]{\bf Conjecture}
\newtheorem{proposition}[theorem]{\bf Proposition}

\newtheorem{lemma}[theorem]{\bf Lemma}
\theoremstyle{definition}

\newenvironment{remark}[1][Remark.]{\begin{trivlist}
		\item[\hskip \labelsep {\bfseries #1}]}{\end{trivlist}}

\newcommand{\Rea}{{\mathbb R}}

\DeclareMathOperator{\rank}{\text{rank}}

\DeclareMathOperator{\supp}{\text{supp}}

\newcommand{\lfrac}[2]{\left\lfloor\frac{#1}{#2}\right\rfloor}

\newcommand{\bv}{\mathbf v}

\usepackage{authblk}

	\title{On the $d$-dimensional algebraic connectivity of graphs}

\author[1]{Alan Lew\thanks{\href{mailto:alan.lew@mail.huji.ac.il}{alan.lew@mail.huji.ac.il}. Alan Lew was partially supported by the Israel Science Foundation grant ISF-2480/20.}}
\author[1]{Eran Nevo\thanks{\href{mailto:nevo@math.huji.ac.il}{nevo@math.huji.ac.il}. Eran Nevo was partially supported by the Israel Science Foundation grant ISF-2480/20.}}
\author[1]{Yuval Peled\thanks{\href{mailto:yuval.peled@mail.huji.ac.il}{yuval.peled@mail.huji.ac.il}.}}
\author[1]{Orit E. Raz\thanks{\href{mailto:oritraz@mail.huji.ac.il}{oritraz@mail.huji.ac.il}.}}

\affil[1]{Einstein Institute of Mathematics,
 Hebrew University, Jerusalem~91904, Israel}
	
	\date{}
\begin{document}

	\maketitle

\begin{abstract}

The $d$-dimensional algebraic connectivity $a_d(G)$ of a graph $G=(V,E)$, introduced by Jord\'an and Tanigawa, is a quantitative measure of the $d$-dimensional rigidity of $G$ that is defined in terms of the eigenvalues of stiffness matrices (which are analogues of the graph Laplacian) associated to mappings of the vertex set $V$ into $\Rea^d$.

Here, we analyze the
$d$-dimensional algebraic connectivity of complete graphs. In particular, we show that, for $d\geq 3$, $a_d(K_{d+1})=1$, and for $n\geq 2d$,
\[
\left\lceil\frac{n}{2d}\right\rceil-2d+1\leq a_d(K_n) \leq \frac{2n}{3(d-1)}+\frac{1}{3}.
\]
\end{abstract}

\section{Introduction}

A \emph{$d$-dimensional framework} is a pair $(G,p)$ consisting of a graph $G=(V,E)$ and a mapping of its vertices $p:V\to \Rea^d$.
A framework $(G,p)$ is called \emph{rigid} if every continuous motion of the vertices that preserves the lengths of all the edges of $G$, preserves in fact the distance between every two vertices of $G$.

In \cite{AR1}, Asimov and Roth introduced the stricter notion of \emph{infinitesimal rigidity}:
For two distinct vertices $u,v\in V$, let $d_{u v}\in \Rea^d$ be defined by
\[
    d_{u v}=\begin{cases}
           \frac{p(u)-p(v)}{\|p(u)-p(v)\|} & \text{ if } p(u)\neq p(v),\\
            0 & \text{ otherwise,}
    \end{cases}
\]
and let $\bv_{u,v}\in \Rea^{d|V|}$ be defined by
\[
   \bv_{u,v} ^t= \kbordermatrix{
     & &  &  & u & & & & v & & & \\
     & 0 & \ldots & 0 & d_{u v} & 0 & \ldots & 0 & d_{v u} & 0 & \ldots & 0}.
\]
Equivalently, $     
\bv_{u,v}= (1_u-1_v)\otimes d_{u v},
$
where $\{1_u\}_{u\in V}$ is the standard basis of $\Rea^{|V|}$ and $\otimes$ denotes the Kronecker product.

The \emph{normalized rigidity matrix} of $(G,p)$, denoted by $R(G,p)$, is the $d|V|\times |E|$ matrix whose columns are the vectors $\bv_{u,v}$ for all $\{u,v\}\in E$.

Assume that $|V|>d$. 
It is known (see \cite{AR1}) that the rank of $R(G,p)$ is at most $d|V|-\binom{d+1}{2}$. The framework $(G,p)$ is called \emph{infinitesimally rigid}\footnote{The usual definition of infinitesimal rigidity is in terms of the unnormalized rigidity matrix, but both definitions are equivalent as scaling of the columns of a matrix does not change its rank.} if the rank of the $R(G,p)$ is exactly 
$d|V|-\binom{d+1}{2}$. 

Infinitesimal rigidity is in general stronger than rigidity, but it is known that both notions coincide for ``generic" embeddings (see \cite{AR1}). Note that for $d=1$, and assuming that $p$ is injective, both notions of rigidity coincide with the notion of graph connectivity.

Recently Jord\'an and Tanigawa~\cite{jordan2020rigidity} (building on Zhu and Hu~\cite{zhu2013quantitative,zhu2009stiffness} who considered the $2$-dimensional case) introduced the following quantitative measure of rigidity:

The \emph{stiffness matrix} $L(G,p)$ is defined by
\[
    L(G,p)=R(G,p) R(G,p)^t \in \Rea^{d|V|\times d|V|}.
\] 
It is easy to check that the rank of $L(G,p)$ equals the rank of $R(G,p)$. Therefore, the kernel of $L(G,p)$ is of dimension at least $\binom{d+1}{2}$, and equality occurs if and only if the framework $(G,p)$ is infinitesimally rigid.

Let $\lambda_i(L(G,p))$ be the $i$-th smallest eigenvalue of $L(G,p)$. The \emph{spectral gap} of  $L(G,p)$ is its minimal non-trivial eigenvalue  $\lambda_{\binom{d+1}{2}+1}(L(G,p))$.
The \emph{$d$-dimensional algebraic connectivity of $G$} is defined by \[
a_d(G)=\sup\left\{ \lambda_{\binom{d+1}{2}+1}(L(G,p)) \middle| \, p: V\to \Rea^d\right\}.
\]
For $d=1$, $a_1(G)$ is the usual algebraic connectivity (a.k.a. spectral gap) of $G$, introduced by Fiedler in \cite{fiedler1973algebraic}.
For general $d$, we always have $a_d(G) \ge 0$ (since $L(G,p)$ is positive semi-definite) and $a_d(G)>0$ if and only if
a generic embedding of $G$ in $\Rea^d$ forms a rigid framework.

Let $K_n$ be the complete graph on $n$ vertices.
In~\cite{jordan2020rigidity}, a lower bound on $a_d(K_n)$ was used to deduce an improved constant in a threshold result for $d$-rigidity~\cite{kiraly2013coherence}: there exists a constant $C_d$ such that if $pn > C_d \log n$ then a graph $G\in G(n,p)$, the Erd\H{o}s-R\'enyi $n$-vertex random graph with edge probability $p$, is asymptotically almost surely $d$-rigid. 
(For a sharp threshold for $d$-rigidity  see our recent~\cite{SharpRigidityThreshold}.)
Their estimate on $C_d$ depended on the value of the spectral gap of the stiffness matrix of the regular $d$-simplex graph $K_{d+1}$, denoted $s_d$, which was conjectured to equal $1$ for $d\geq 3$.  

Motivated by these results, we study in this paper the $d$-dimensional algebraic connectivity of complete graphs. It is well known and easy to check that $a_1(K_n)=n$.
For $d=2$, it was shown by Jord\'an and Tanigawa \cite[Theorem 4.4]{jordan2020rigidity}, based on a result by Zhu \cite{zhu2013quantitative}, that $a_2(K_n)=n/2$ for all $n\geq 3$ (see also Proposition \ref{prop:2d}). For $d\geq 3$ the situation is more complicated, and the only previously known result is the lower bound $a_d(K_n)\geq s_d d n/ (2 (d+1)^2)-d$, proved by Jord\'an and Tanigawa in \cite[Thm. 5.2]{jordan2020rigidity}\footnote{The bound as stated in \cite{jordan2020rigidity} is incorrect, the bound stated here is obtained after fixing a typo in \cite[Lemma 5.4]{jordan2020rigidity}.}.

Here, we first focus on the case $n=d+1$. Let $p^{\triangle}:V(K_{d+1})\to \Rea^d$ denote the regular simplex embedding (that is, the vertices of $K_{d+1}$ are mapped bijectively to the vertices of a regular simplex in $\Rea^d$), and denote 
\[s_d=\lambda_{\binom{d+1}{2}+1}(L(K_{d+1},p^{\triangle})).
\]
We prove the following.

\begin{restatable}{theorem}{simplexspectrum}\label{thm:simplex_spectrum}
The spectrum of $L(K_{d+1},p^{\triangle})$ is
\[
    \left\{0^{(d(d+1)/2)},1^{((d+1)(d-2)/2)},\frac{d+1}{2}^{(d)},d+1^{(1)}\right\}.
\]
\end{restatable}
(The superscript $(m)$ indicates multiplicity $m$ of the corresponding eigenvalue, here and throughout the paper.)

This settles a conjecture of Jord\'an and Tanigawa \cite[Conj. 1]{jordan2020rigidity}. In particular, we obtain $s_d=1$ for $d\ge 3$. (Note that $s_2=3/2$.)
Further, we show that this is the largest possible spectral gap for a framework $(K_{d+1},p)$. That is, 
\begin{restatable}{theorem}{connectivityofsimplex}
\label{thm:simplex}
    For $d\geq 3$, $a_d(K_{d+1})=1$.
\end{restatable}
However, for $d\ge 3$, $p^{\triangle}$ is not the only embedding that achieves the maximum value $a_d(K_{d+1})=1$, see Proposition~\ref{prop:nonregular}.

Next we consider (balanced) Tur{\'a}n graphs:
Let $r,n$ be positive integers such that $r$ divides $n$.
Let $V_1,\ldots,V_r$ be pairwise disjoint sets such that $|V_i|=n/r$ for all $i\in[r]$. Let $V=V_1\cup\cdots\cup V_r$ and
$E=\cup_{i\neq j\in[r]} \{\{u,v\}:\, u\in V_i, v\in V_j\}$. The graph $T(n,r)=(V,E)$ is called
a Tur{\'a}n graph, or the complete balanced $r$-partite graph on $n$ vertices.

For $r=d+1$ let $q^{\triangle}: V\to \Rea^d$ denote the mapping into the vertices of a regular $d$-simplex, such that the preimage of each vertex of the simplex equals $V_i$ for a different $i\in [d+1]$.  
We compute the spectrum of $L(T(n,d+1),q^{\triangle})$:

\begin{restatable}{theorem}{simplexrepeated}\label{thm:simplex_repeated}
Let $d\geq 2$ and $n\geq d+1$ such that $n$ is divisible by $d+1$. Then, the spectrum of $L(T(n,d+1),q^{\triangle})$ is
\[
    \left\{
    0^{(d(d+1)/2)},\frac{n}{2(d+1)}^{((n-d-1)(d-1))},\frac{n}{d+1}^{((d-2)(d+1)/2)},\frac{n}{2}^{(n-1)},
    n^{(1)}
    \right\}.
\]
\end{restatable}

In particular, 
its spectral gap for $n\geq 2(d+1)$ is:
\[
    \lambda_{\binom{d+1}{2}+1}(L(T(n,d+1),q^{\triangle}))= \frac{n}{2(d+1)}.
\]
This improves upon the lower bound $\lambda_{\binom{d+1}{2}+1}(L(T(n,d+1),q^{\triangle}))\geq \frac{d n}{2(d+1)^2}$ obtained  in~\cite{jordan2020rigidity} (after fixing a typo and plugging $s_d=1$).

Similarly, for $r=2d$ let $q^{\diamond}:V\to \Rea^d$ denote the mapping into the vertices of a regular $d$-crosspolytope, such that the preimage of each vertex of the crosspolytope equals $V_i$ for a different $i\in [2d]$ (recall that the regular $d$-crosspolytope is defined as the convex hull of the set $\{\pm e_1,\ldots, \pm e_d\}$, where $\{e_1,\ldots,e_d\}$ is the standard basis of $\Rea^d$). We compute the spectrum of $L(T(n,2d),q^{\diamond})$:

\begin{restatable}{theorem}{crosspolytopespectrum}\label{thm:crosspolytope_spectrum}
Let $d\geq 2$ and $n\geq 2d$ such that $n$ is divisible by $2d$. 
Then, the spectrum of $L(T(n,2d),q^{\diamond})$ is
\[
    \left\{
    0^{(d(d+1)/2)},\frac{n}{2d}^{(n(d-1)-d^2)},\frac{n}{d}^{(d(d-1)/2)},\frac{n}{2}^{(n-1)},
    n^{(1)}
    \right\}.
\]
\end{restatable}
In particular, 
its spectral gap is:
\[
    \lambda_{\binom{d+1}{2}+1}(L(T(n,2d),q^{\diamond}))= \frac{n}{2d}.
\]

The proofs of Theorems \ref{thm:simplex_repeated} and \ref{thm:crosspolytope_spectrum} follow by computing the eigenbases the corresponding stiffness matrices.
As a corollary of Theorem \ref{thm:crosspolytope_spectrum}, we obtain
the following lower bound on $a_d(K_n)$:

\begin{restatable}{theorem}{lowerbound}\label{thm:lower_bound}
Let $d\geq 3$ and $n\geq 2d$. If $n$ is divisible by $2d$, then
\[
      a_d(K_n)\geq \frac{n}{2d}.
\]
For general $n\geq 2d$, we have
$
    a_d(K_n)\geq \left\lceil\frac{n}{2d}\right\rceil-2d+1.
$
\end{restatable} 
This improves the previously known lower bound. Finally, we prove an upper bound on the $d$-dimensional algebraic connectivity of the complete graph:

\begin{restatable}{theorem}{upperbound}\label{cor:upperbound}
Let $d\geq 3$ and $n\geq d+1$. Then, 
\[
    a_d(K_n)\leq \frac{2n}{3(d-1)}+\frac{1}{3}\,.
\]
\end{restatable}
This follows by proving a lower bound on the sum of the $n$ largest eigenvalues of $L(K_n,p)$ for every embedding $p$ into $\Rea^d$ (see Lemma \ref{lem:one_third}).

Most of our results rely on analyzing the \emph{lower stiffness matrix} of the framework $(G,p)$, defined by
\[
    L^{-}(G,p)=R(G,p)^t R(G,p) \in \Rea^{|E|\times |E|}.
\] 
It is easy to check that $\rank(L(G,p))=\rank(L^{-}(G,p))=\rank(R(G,p))$, and that the non-zero eigenvalues of $L(G,p)$ are the same as those of $L^{-}(G,p)$, namely also with same multiplicities.

\textbf{Outline}: In Section~\ref{sec:prelim} we describe $L^{-}(G,p)$ explicitly and apply Cauchy's Interlacing Theorem to relate the spectra of $L(G,p)$ and $L(G\setminus{e},p)$ where $e$ is an edge in $G$.
In Section~\ref{sec:spec_simplex} we compute the spectrum of $L(K_{d+1},p^{\triangle})$ and determine the value of $a_d(K_{d+1})$.
In Section~\ref{sec:Turan} we find the spectrum and eigenbasis of $L(T(n,d+1),q^{\triangle})$, improving on the estimate in~\cite{jordan2020rigidity} for its spectral gap, and similarly we find the spectrum and eigenbasis of $L(T(n,2d),q^{\diamond})$, concluding a lower bound on $a_d(K_n)$.
In Section~\ref{sec:top-n-ev} we lower bound the sum of the largest $n$ eigenvalues of $L(K_n,p)$ for every embedding $p$, and use it to prove an upper bound on $a_d(K_n)$.
We conclude in Section~\ref{sec:conclude} with open problems and conjectures on $a_d(G)$ for graphs of interest.

\section{The lower stiffness matrix}\label{sec:prelim}

We start with the following explicit description of $L^{-}(G,p)$:

\begin{lemma}\label{lemma:down_laplacian}
Let $G=(V,E)$ be a graph and let $p:V\to \Rea^d$. Let $e_1,e_2\in E$. Then,
\[
    L^{-}(G,p)(e_1,e_2)= \begin{cases}
                    2 & \text{ if } e_1=e_2=\{i,j\} \text{ and } p(i)\neq p(j),\\
  \cos(\theta(e_1,e_2)) & \text{ if } |e_1\cap e_2|=1,\\
                    0 & \text{ otherwise,}
                    \end{cases} 
\]
where, for $e_1=\{i,j\}$ and $e_2=\{i,k\}$, 
$\theta(e_1,e_2)$ is the angle between $d_{ij}$ and $d_{ik}$; that is, $\cos(\theta(e_1,e_2))=d_{ij}\cdot d_{ik}$ (note that, by convention, if $d_{ij}=0$ or $d_{ik}=0$, then $\cos(\theta(e_1,e_2))=0$).
\end{lemma}
\begin{proof}
For convenience, we identify the vertex set of $G$ with the set $[n]$. Let $e_1=\{i,j\}$ and $e_2=\{k,l\}$, where $i<j$ and $k<l$. Let $\{1_i\}_{i\in V}$ be the standard basis for $\Rea^{|V|}$, and $\{1_e\}_{e\in E}$ be the standard basis for $\Rea^{|E|}$.
Then,
\[
    R(G,p)1_{e_1}= (1_i-1_j)\otimes d_{ij},
\]
and similarly
\[
    R(G,p)1_{e_2}= (1_k-1_l)\otimes d_{kl}.
\]
Note that
\[
L^{-}(G,p)(e_1,e_2)= (R(G,p) 1_{e_1})^t (R(G,p)1_{e_2})= ((1_i-1_j)\cdot(1_k-1_l))(d_{ij}\cdot d_{kl}).
\]
If $e_1=e_2$, we obtain
\[
L^{-}(G,p)(e_1,e_2)= \|1_i-1_j\|^2 \cdot \|d_{ij}\|^2=\begin{cases} 2 & \text{ if } p(i)\neq p(j),\\
0 & \text{ otherwise.}
\end{cases}
\]
If $e_1\cap e_2=\emptyset$, then
\[
L^{-}(G,p)(e_1,e_2)=  ((1_i-1_j)\cdot(1_k-1_l))(d_{ij}\cdot d_{kl})=0\cdot (d_{ij}\cdot d_{kl})=0.
\]
Finally, assume $|e_1\cap e_2|=1$. Then, either $i=k$, or $i=l$, or $j=k$ or $j=l$. If $i=k$, then
\[
L^{-}(G,p)(e_1,e_2)=  1\cdot (d_{ij}\cdot d_{il})= \cos(\theta(e_1,e_2)).
\]
If $i=l$, then
\[
L^{-}(G,p)(e_1,e_2)=  (-1)\cdot (d_{ij}\cdot d_{ki})= d_{ij}\cdot d_{ik}= \cos(\theta(e_1,e_2)).
\]
The other two cases follow similarly.
\end{proof}

\begin{remark} In \cite{aryankia2021spectral}, the lower stiffness matrix $L^{-}(G,p)$ was studied in the special case where $G=K_3$, the complete graph on three vertices, and $p$ is an embedding of the vertices in $\Rea^2$.
\end{remark}

\subsection{Interlacing of spectra}

We will use the following special case of Cauchy's interlacing theorem:
\begin{theorem}[See e.g. \cite{brouwer2011spectra}]\label{thm:cauchy}
Let $A$ be a real symmetric matrix of size $n\times n$ and $B$ a principal submatrix of $A$ of size $(n-1)\times(n-1)$. Let $\mu_1\geq \mu_2\geq  \cdots\geq \mu_n$ be the eigenvalues of $A$ and $\mu'_1\geq\mu'_2\geq \cdots \geq \mu'_{n-1}$ be the eigenvalues of $B$. Then, for $1\leq i\leq n-1$, we have
\[
    \mu_i \geq \mu'_i\geq \mu_{i+1}.
\]
\end{theorem}

We obtain the following interlacing result, generalizing a known result for graph Laplacians (see e.g. \cite[Theorem 13.6.2]{godsil2001algebraic}).

\begin{theorem}\label{thm:edge_removal_interlacing}
Let $G=(V,E)$ be a graph with $|V|=n$, and let $p:V\to \Rea^d$. Let $e\in E$, and let $G\setminus e=(V,E\setminus\{e\})$. Let $\lambda_1\leq \lambda_2\leq \cdots\leq \lambda_{dn}$ be the eigenvalues of $L(G,p)$ and $\lambda'_1\leq \cdots \leq \lambda'_{dn}$ be the eigenvalues of $L(G\setminus e,p)$. Let $\lambda_0=0$. Then, we have
\[
    \lambda_{i-1} \leq \lambda'_{i} \leq \lambda_{i},
\]
for all $1\leq i\leq dn$.
\end{theorem}
\begin{proof}
Let $\mu'_1\geq \cdots\geq \mu'_{|E|}$ be the eigenvalues of $L^{-}(G,p)$ and $\mu_1\geq \cdots\geq \mu_{|E|-1}$ be the eigenvalues of $L^{-}(G\setminus e,p)$. 

 Note that $L^{-}(G\setminus\{e\},p)$ is a principal submatrix of $L^{-}(G,p)$, therefore by Theorem \ref{thm:cauchy}, we have
\[
    \mu_i \geq \mu'_i \geq \mu_{i+1}
\]
for $i=1,\ldots,|E|-1$. 

For $i>|E|-1$, let $\mu'_i=0$, and for $i>|E|$, let $\mu_i=0$. Then, we have
\[
    \mu_i \geq \mu'_i \geq \mu_{i+1}
\]
for all $i$.
Since $L(G,p)$ and $L^{-}(G,p)$ have the same positive eigenvalues, we have for all $i=1,\ldots,d n$
\[
    \lambda_{dn+1-i}=\mu_{i},
\]
and similarly
\[
 \lambda'_{dn+1-i}=\mu'_{i}.
\]
Therefore, we have
\[
    \lambda_{dn+1-i} \geq \lambda'_{dn+1-i} \geq \lambda_{dn-i}
\]
for all $i=1,\ldots,dn$ (using $\lambda_0=0$).
So, for $j=1,\ldots,dn$, we obtain
\[
    \lambda_{j-1} \leq \lambda'_{j} \leq \lambda_{j},
\]
as wanted.
\end{proof}

As an application of Theorem \ref{thm:edge_removal_interlacing}, we show that restricting attention to maps $p:V\to \Rea^d$ that are embeddings (i.e injective) does not affect the $d$-dimensional algebraic connectivity of a graph $G=(V,E)$.

\begin{lemma}\label{lemma:a_d_for_embeddings}
Let $G=(V,E)$ and $d\geq 1$. Then,
\[
a_d(G)=\sup\left\{ \lambda_{\binom{d+1}{2}+1}(L(G,p)) \middle| \, p: V\to \Rea^d,\, p \text{ is injective}\right\}.
\]
\end{lemma}
\begin{proof}
Let
\[
\tilde{a}_d(G)=\sup\left\{ \lambda_{\binom{d+1}{2}+1}(L(G,p)) \middle| \, p: V\to \Rea^d,\, p \text{ is injective}\right\}.
\]
Clearly, $\tilde{a}_d(G)\leq a_d(G)$. 
In the other direction, let $p:V\to\Rea^d$. We will show that for any $\epsilon>0$ there exists a $p':V\to \Rea^d$ such that $p'$ is injective and
\[
    \lambda_{\binom{d+1}{2}+1}(L(G,p')) > \lambda_{\binom{d+1}{2}+1}(L(G,p))-\epsilon.
\]
Let $E'=\{\{u,v\}\in E:\, p(u)\neq p(v)\}$ and let
$G'=(V,E')$. Note that $L(G,p)=L(G',p)$, and in particular $\lambda_{\binom{d+1}{2}+1}(L(G,p))=\lambda_{\binom{d+1}{2}+1}(L(G',p))$.

By Lemma \ref{lemma:down_laplacian}, for $p'$ in a neighborhood of $p$, the entries of the lower stiffness matrix $L^{-}(G',p')$ are continuous functions of $p'$. Therefore, the spectral gap $\lambda_{\binom{d+1}{2}+1}(L(G',p'))$ is also continuous in $p'$ (as a root of the characteristic polynomial of $L^{-}(G',p'))$. That is, there exists $\delta>0$ such that if $\|p(u)-p'(u)\|<\delta$ for all $u\in V$, then
\[
    \left|\lambda_{\binom{d+1}{2}+1}(L(G',p'))-\lambda_{\binom{d+1}{2}+1}(L(G',p))\right|<\epsilon.
\]
Now, let $p':V\to\Rea^d$ be an embedding satisfying $\|p(u)-p'(u)\|< \delta$ for all $u\in V$. Then, we have
\[
\lambda_{\binom{d+1}{2}+1}(L(G',p'))> \lambda_{\binom{d+1}{2}+1}(L(G',p))-\epsilon = \lambda_{\binom{d+1}{2}+1}(L(G,p))-\epsilon.
\]
Finally, by Theorem \ref{thm:edge_removal_interlacing}, we obtain
\[
\lambda_{\binom{d+1}{2}+1}(L(G,p')) \geq \lambda_{\binom{d+1}{2}+1}(L(G',p'))> \lambda_{\binom{d+1}{2}+1}(L(G,p))-\epsilon.
\]
Thus, $\tilde{a}_d(G)\geq a_d(G)$, as wanted.
\end{proof}

\section{The $d$-dimensional algebraic connectivity of the simplex graph}
\label{sec:spec_simplex}

It is known (see \cite{jordan2020rigidity}) that $a_1(K_{n})=n$ and $a_2(K_n)=n/2$. In particular, $a_2(K_3)=3/2$. In this section we prove the following:

\connectivityofsimplex*

\subsection{Lower bound: $a_d(K_{d+1})\ge 1$}

Recall that $p^{\triangle}:V\to \Rea^d$ is the embedding that maps each vertex of $K_{d+1}$ to one of the vertices of the regular $d$-dimensional simplex.
The lower bound follows from the following result, conjectured in~\cite[Conj.1]{jordan2020rigidity}:

\simplexspectrum*
\begin{proof}
Let $K_{d+1}=(V,E)$, where $V=[d+1]$ and $E=\binom{[d+1]}{2}$. Since the angle between every two intersecting edges of the regular simplex is $60\degree$, we have, by Lemma \ref{lemma:down_laplacian},
\[
    L^{-}(K_{d+1},p^{\triangle})(e_1,e_2)= \begin{cases}
                    2 & \text{ if } e_1=e_2,\\
                    \frac{1}{2} & \text{ if } |e_1\cap e_2|=1,\\
                    0 & \text{ otherwise,}
                    \end{cases} 
\]
for every $e_1,e_2\in E$. We can write
\[
    L^{-}(K_{d+1},p^{\triangle})= I+ \frac{1}{2} Q,
\]
where $Q\in \Rea^{|E|\times|E|}$ is defined by
\[
    Q(e_1,e_2)= \begin{cases}
                    2 & \text{ if } e_1=e_2,\\
                   1 & \text{ if } |e_1\cap e_2|=1,\\
                    0 & \text{ otherwise,}
                    \end{cases} 
\]
for every $e_1,e_2\in E$. 

Let $M\in \Rea^{(d+1)\times|E|}$ be the signless incidence matrix of $K_{d+1}$, defined by
\[
    M(i,e)=\begin{cases}
                    1 & \text{ if } i\in e,\\
                    0 & \text{ otherwise,}
                    \end{cases} 
\]
for $i\in V=[d+1]$ and $e\in E$. Then, we have
\[
    Q= M^{t} M.
\]
Let
\[
    \tilde{Q}= M M^{t}\in \Rea^{(d+1)\times(d+1)}.
\]
The matrix $\tilde{Q}$ is the signless Laplacian of $K_{d+1}$, namely:
\[
    \tilde{Q}(i,j)=\begin{cases}
                    d & \text{ if } i=j,\\
                    1 & \text{ otherwise.}
                    \end{cases} 
\]
Therefore, $\tilde{Q}=(d-1)I+J$, where $J$ is the all-ones matrix. The spectrum of $J$ is $\{0^{(d)},d+1^{(1)}\}$; therefore, the spectrum of $\tilde{Q}$ is $\{d-1^{(d)},2d^{(1)}\}$. Since the non-zero eigenvalues of $Q$ and $\tilde{Q}$ are the same, the spectrum of $Q$ is $\{0^{((d-2)(d+1)/2)},d-1^{(d)},2d^{(1)}\}$.
Thus, the spectrum of $L^{-}(K_{d+1},p^{\triangle})$ is 
\[
\left\{1^{((d-2)(d+1)/2)},\frac{d+1}{2}^{(d)},d+1^{(1)}\right\}.
\]
Finally, since the non-zero eigenvalues of $L^{-}(K_{d+1},p^{\triangle})$ and $L(K_{d+1},p^{\triangle})$ are the same, the spectrum of $L(K_{d+1},p^{\triangle})$ is
\[
\left\{0^{(d(d+1)/2)},1^{((d-2)(d+1)/2)},\frac{d+1}{2}^{(d)},d+1^{(1)}\right\},
\]
as wanted.
\end{proof}

As a consequence of Theorem \ref{thm:simplex}, we obtain $a_d(K_{d+1})\geq 1$ for all $d\geq 3$. We are left to show that $a_d(K_{d+1})\leq 1$.
Before doing that, let us remark that for $d\ge 3$ there are embeddings $p\neq p^{\triangle}$ such that $\lambda_{\binom{d+1}{2}+1}(L(K_{d+1},p))=1$:
\begin{proposition}\label{prop:nonregular}
Let $p_1,p_2,p_3\in \Rea^2$ be the vertices of an equilateral triangle with sides of length $1$ centered at the origin. For $h\geq 0$, let $p_h:[4]\to \Rea^3$ be defined by
\[
    p_h(i)=\begin{cases}
            (p_i,0) &  \text{ if } i\in[3],\\
            (0,h) & \text{ if } i=4.
        \end{cases}
\]
Then, the spectrum of $L(K_4,p_h)$ is
\[
    \left\{0^{(6)},1^{(2)},\left(3-\left(h^2+\frac{1}{3}\right)^{-1}\right)^{(1)},\left(\frac{3}{2}+\frac{1}{2}\left(h^2+\frac{1}{3}\right)^{-1}\right)^{(2)},4^{(1)}\right\}.
\]
In particular, for $h\geq \frac{1}{\sqrt{6}}$, we have $\lambda_7(L(K_4,p_h))=1$.
\end{proposition}
\begin{proof}
Denote the edges of the tetrahedron formed by the image of $p_h$ by $e_{ij}$, $1\leq i\neq j\leq 4$.
Let $\ell$ be the length of the edges $e_{i4}$ (for $i\in[3]$). Note that $\ell=\sqrt{h^2+\frac{1}{3}}$. It is easy to check that for three distinct indices $i,j,k$
\[
    \cos(\theta(e_{ij},e_{jk}))=\begin{cases}
    \frac{1}{2} & \text{ if } i,j,k\in [3],\\
    \frac{1}{2\ell} & \text{ if } i,j\in [3], k=4,\\
    1-\frac{1}{2\ell^2} & \text{ if } i,k\in[3], j=4.
    \end{cases}
\]
Therefore, by Lemma \ref{lemma:down_laplacian}, we have
\[
    L^{-}(K_4,p_h)=\begin{pmatrix}
  2 & \frac{1}{2} & \frac{1}{2} & \frac{1}{2\ell} & \frac{1}{2\ell} & 0\\[5pt]
  \frac{1}{2} & 2 & \frac{1}{2} & \frac{1}{2\ell} & 0 & \frac{1}{2\ell}\\[5pt]
  \frac{1}{2} & \frac{1}{2} & 2 & 0 & \frac{1}{2\ell} & \frac{1}{2\ell}\\[5pt]
  \frac{1}{2\ell} & \frac{1}{2\ell} & 0 & 2 & 1-\frac{1}{2\ell^2} & 1-\frac{1}{2\ell^2}\\[5pt]
  \frac{1}{2\ell} & 0 & \frac{1}{2\ell} & 1-\frac{1}{2\ell^2} & 2 & 1-\frac{1}{2\ell^2}\\[5pt]
  0 & \frac{1}{2\ell} & \frac{1}{2\ell} & 1-\frac{1}{2\ell^2} & 1-\frac{1}{2\ell^2} & 2
\end{pmatrix}.
\]
The spectrum of $L^{-}(K_4,p_h)$ can now be computed with the help of a computer algebra system. We obtain the following eigenvalues:
\[
    1^{(2)},\left(3-\ell^{-2}\right)^{(1)},\left(\frac{3}{2}+\frac{1}{2}\ell^{-2}\right)^{(2)},4^{(1)}.
\]
Since the non-zero eigenvalues of $L^{-}(K_4,p_h)$ are the same as those of $L(K_4,p_h)$, and $\ell^2=h^2+\frac{1}{3}$, the spectrum of $L(K_4,p_h)$ is  \[
  \left\{  0^{(6)},1^{(2)},\left(3-\left(h^2+\frac{1}{3}\right)^{-1}\right)^{(1)},\left(\frac{3}{2}+\frac{1}{2}\left(h^2+\frac{1}{3}\right)^{-1}\right)^{(2)},4^{(1)}\right\},
\]
as wanted. Finally, note that for $h\geq \frac{1}{\sqrt{6}}$ we obtain 
\[
3-(h^2+\frac{1}{3})^{-1}\geq 3-2 =1,
\]
and therefore $\lambda_7(K_4,p_h)=1$.
\end{proof}

\subsection{Upper bound: $a_d(K_{d+1})\le 1$}
First, we prove the $3$-dimensional case:

\begin{proposition}
\[
    a_3(K_4)= 1.
\]
\end{proposition}
\begin{proof}
Let $p:V\to \Rea^3$. Note that, since $|E|=6=dn-\binom{d+1}{2}$ (for $d=3$ and $n=4$), 
\[
    \lambda_7(L(K_4,p))=\lambda_{\binom{3+1}{2}+1}(L(K_4,p))
\]
is equal to the minimal eigenvalue of $L^{-}(K_4,p)$.
Thus, for every $0\neq x\in \Rea^6$
\[
 \lambda_7(L(K_4,p)) \leq \frac{x^t L^{-}(K_4,p) x}{\|x\|^2}.
\]
Let $e_{1},e_1',e_2,e_2',e_3,e_3'$ be the edges of $K_4$, such that $e_i\cap e_i'=\emptyset$ for all $i$.
Let $\ell_1,\ell_1',\ell_2,\ell'_2,\ell_3,\ell'_3$ be the lengths of the images of $e_{1},e_1',e_2,e_2',e_3,e_3'$ (that is, if $e_i=\{u,v\}$, $\ell_i=\|p(u)-p(v)\|$).

Assume without loss of generality that $\ell_3^2+\ell_3^{'2} \leq \min\{\ell_1^2+\ell_1^{'2} ,\ell_2^2+\ell_2^{'2}\}$. Let
\[
    x= \ell_1 1_{e_1}+\ell_1' 1_{e_1'}-\ell_2 1_{e_2}-\ell_2' 1_{e_2'}.
\]
Then, $\|x\|^2=\ell_1^2+\ell_1^{'2}+\ell_2^2+\ell_2^{'2}$, and
\begin{multline*}
    x^t L^{-}(K_4,p) x
    =2\|x\|^2 - 2(\ell_1\ell_2\cos(\theta(e_1,e_2))+\ell_1\ell_2'\cos(\theta(e_1,e'_2))\\+\ell_1'\ell_2\cos(\theta(e'_1,e_2)) 
    +\ell_1'\ell_2'\cos(\theta(e'_1,e'_2))).
\end{multline*}
By the law of cosines, we obtain
\[
    x^t L^{-}(K_4,p) x
    =2\|x\|^2 +2(\ell_3^2+\ell_3^{'2}- \ell_1^2-\ell_1^{'2}-\ell_2^2-\ell_2^{'2}) = 2(\ell_3^2+\ell_3^{'2}).
\]
So,
\[
\frac{x^t L^{-}(K_4,p) x}{\|x\|^2}= \frac{2(\ell_3^2+\ell_3^{'2})}{\ell_1^2+\ell_1^{'2}+\ell_2^2+\ell_2^{'2}}\leq 1.
\]
Therefore, we obtain
\[
 \lambda_7(L(K_4,p)) \leq 1.
\]
\end{proof}

Finally, we show that the bound for general $d$ follows from the $d=3$ case:

\begin{proposition}
For all $d\geq 4$,
\[
    a_d(K_{d+1})\leq a_3(K_4)=1.
\]
\end{proposition}
\begin{proof}
Let $K_{d+1}=(V,E)$, where $V=[d+1]$ and $E=\binom{[d+1]}{2}$.
We will show that for every $d$,
\[
    a_d(K_{d+1})\leq a_{d-1}(K_{d}).
\]
Let $G$ be the graph obtained by adding an isolated vertex $v$ to $K_d$. Then, $G$ is obtained from $K_{d+1}$ by removing the $d$ edges containing $v$. Therefore, for every $p:V\to \Rea^d$, we have
by Theorem \ref{thm:edge_removal_interlacing},
\[
    \lambda_{\binom{d+1}{2}+1}(L(K_{d+1},p)) \leq \lambda_{\binom{d+1}{2}+d+1}(L(G,p)).
\]
Let $H$ be an affine hyperplane containing $p(V\setminus\{v\})$. Identify $H$ with $\Rea^{d-1}$, and let $p'=p|_{V\setminus\{v\}}: V\setminus\{v\}\to\Rea^{d-1}$.
Note that $L^{-}(G,p)=L^{-}(K_d,p')$; therefore,
the non-zero eigenvalues of $L(G,p)$ and $L(K_d,p')$ are the same. Since $L(K_d,p')\in \Rea^{(d-1)d \times (d-1)d}$ and $L(G,p)\in \Rea^{d(d+1)\times d(d+1)}$, this means in particular that
\[
\lambda_{\binom{d}{2}+1}(L(K_d,p')) =\lambda_{\binom{d}{2}+2d+1}(L(G,p))=\lambda_{\binom{d+1}{2}+d+1}(L(G,p)).
\]  
So, we obtain
\[
    \lambda_{\binom{d+1}{2}+1}(L(K_{d+1},p)) \leq \lambda_{\binom{d}{2}+1}(L(K_d,p')) \leq a_{d-1}(K_d).
\]
Since this holds for every $p:V\to \Rea^d$, we obtain
\[
   a_d(K_{d+1}) \leq a_{d-1}(K_d),
\]
as wanted. Therefore, the bound
\[
   a_d(K_{d+1}) \leq a_{3}(K_4)=1
\]
follows by induction on $d$.
\end{proof}

\section{Spectra of Tur{\'a}n graphs
$T(n,d+1)$ and $T(n,2d)$}\label{sec:Turan}

Recall that $T(n,r)$ is the complete balanced $r$-partite graph with $n$ vertices. For $r=d+1$, we denote by $q^{\triangle}: [n]\to \Rea^d$ the mapping that maps the vertices of each of the $d+1$ parts of $T(n,d+1)$ to one of the vertices of the regular $d$-dimensional simplex. Similarly, we denote by $q^{\diamond}:[n]\to \Rea^d$ the mapping that maps the vertices of each of the $2d$ parts of $T(n,2d)$ to one of the vertices of the regular $d$-dimensional crosspolytope.
In this section we determine the spectra of  $L(T(n,d+1),q^{\triangle})$ and $L(T(n,2d),q^{\diamond})$. 
In fact, in each case we provide a basis of $\Rea^{E}$ consisting of eigenvectors of the corresponding lower stiffness matrix.

As a consequence, we obtain a lower bound on the $d$-dimensional algebraic connectivity of $K_n$ (Theorem \ref{thm:lower_bound}).

We begin with the following result of Jord\'an and Tanigawa:

\begin{lemma}[{\cite[Lemma 4.3]{jordan2020rigidity}}]\label{lemma:n_eigenvalue}
Let $p:[n]\to \Rea^d$. If $p$ is not constant, then the largest eigenvalue of $L(K_n,p)$ is $n$.
\end{lemma}
In \cite{jordan2020rigidity} it was shown that $p$ itself (when considered as a vector in $\Rea^{d n}$) is an eigenvector of $L(K_n,p)$ with eigenvalue $n$. The corresponding eigenvector for $L^{-}(K_n,p)$ is $\phi=R(G,p)^{t} p$, which satisfies
\[
    \phi(\{i,j\})= \|p(i)-p(j)\|
\]
for each $i\neq j\in[n]$.

The following result shows that for mappings $p$ satisfying certain ``spherical symmetry", $n/2$ is an eigenvalue of $L(K_n,p)$ of high multiplicity:

\begin{proposition}\label{prop:large_eigenvalues_balanced_p}
Let $p:[n]\to\Rea^d$ such that $\|p(i)\|=1$ for all $i\in[n]$ and $\sum_{i=1}^n p(i)=0$. Assume that the image of $p$ is of size at least $3$. Then, $n/2$ is an eigenvalue of $L(K_n,p)$ with multiplicity at least $n-1$.
\end{proposition}
\begin{proof}
Since the non-zero eigenvalues of $L(K_n,p)$ and $L^{-}(K_n,p)$ are the same, it is enough to look at $L^{-}(K_n,p)$. 

Let $f\in \Rea^n$ such that $\sum_{i=1}^n f_i=0$. 
For every $i\neq j\in[n]$, let $\ell_{ij}=\|p(i)-p(j)\|$. Let $E=E(K_n)=\{\{i,j\}:\, 1\leq i<j\leq n\}$, and let $\phi_f \in \Rea^E$ be defined by
\[
    \phi_f(\{i,j\})=(f_i+f_j) \ell_{ij}.
\]
We will show that $\phi_f$ is an eigenvector of $L^{-}(K_n,p)$ with eigenvalue $n/2$.

For $i,j,k\in[n]$, let $\theta_{ijk}$ be the angle between $p(i)-p(j)$ and $p(k)-p(j)$.
By the law of cosines, we have
\[
    \ell_{ij}\ell_{jk}\cos(\theta_{ijk})=\frac{1}{2}\left(\ell_{ij}^2+\ell_{jk}^2-\ell_{ik}^2\right).
\]
Let $I'\in \Rea^{E\times E}$ be a diagonal matrix with elements $I'_{e,e}=1$ if $e=\{i,j\}$ where $p(i)\neq p(j)$ and $I'_{e,e}=0$ otherwise.
Let $A=L^{-}(K_n,p)-2I'$.
Let $e=\{i,j\}\in E$. First, assume that $p(i)\neq p(j)$.
Then, by Lemma \ref{lemma:down_laplacian}, we have
\begin{align*}
A \phi_f(e)&= \sum_{k\neq i,j} \left(\cos(\theta_{ijk}) \phi_f(\{j,k\}) +\cos(\theta_{jik})\phi_f(\{i,k\})\right)
\\
&=\sum_{k\neq i,j} \left(\ell_{jk}\cos(\theta_{ijk})(f_j+f_k)  +\ell_{ik}\cos(\theta_{jik})(f_i+f_k)\right)
\\
&=\sum_{k\neq i,j} \left(\frac{\ell_{ij}^2+\ell_{jk}^2-\ell_{ik}^2}{2\ell_{ij}}(f_j+f_k)  +\frac{\ell_{ij}^2+\ell_{ik}^2-\ell_{jk}^2}{2\ell_{ij}}(f_i+f_k)\right)
\\
&=\sum_{k\neq i,j} \ell_{ij} f_k
+ \sum_{k\neq i,j} (f_i+f_j)\frac{\ell_{ij}}{2} +\frac{f_j-f_i}{2\ell_{ij}} \sum_{k\neq i,j}(\ell_{jk}^2-\ell_{ik}^2).
\end{align*}
Note that, since $\sum_{k=1}^n f_k=0$, we have
\[
    \sum_{k\neq i,j} f_k = -(f_i+f_j).
\]
Also, since $\|p(x)\|=1$ for all $x\in[n]$, for all $x,y\in[n]$ we have
\[
    \ell_{xy}^2= \|p(x)\|^2+\|p(y)\|^2-2 p(x)\cdot p(y)= 2-2 p(x)\cdot p(y).
\]
So, since $\sum_{k=1}^n p(k)=0$, we obtain
\begin{multline*}
    \sum_{k\neq i,j} (\ell_{jk}^2-\ell_{ik}^2) = 2(p(i)-p(j))\cdot\sum_{k\neq i,j} p(k)
    \\
    = 2(p(j)-p(i))\cdot(p(i)+p(j))= 2(\|p(j)\|^2-\|p(i)\|^2)=0.
\end{multline*}
Therefore, we obtain
\[
A \phi_f(e)= -(f_i+f_j)\ell_{ij}+\frac{n-2}{2}(f_i+f_j) \ell_{ij}= \frac{n-4}{2} \phi_f(e).
\]
So, $L^{-}(K_n,p)\phi_{f}(e)=(n/2)\phi_f(e)$.
Now, assume $p(i)=p(j)$. Note that $\phi_f(e)=(f_i+f_j)\ell_{i j}=0$. Then, since $\cos(\theta_{i j k})=0$ and $\cos(\theta_{j i k})=0$ for all $k\neq i,j$, we obtain
\[
A \phi_f(e)= \sum_{k\neq i,j} \left(\cos(\theta_{ijk}) \phi_f(\{j,k\}) +\cos(\theta_{jik})\phi_f(\{i,k\}\right)=0,
\]
and therefore $L^{-}(K_n,p)\phi_{f}(e)=0=(n/2)\phi_f(e)$.
Thus, $\phi_f$ is an eigenvector of $L^{-}(K_n,p)$ with eigenvalue $n/2$.

Finally, we will show that the dimension of the subspace
\[
    U=\left\{\phi_f:\, f\in \Rea^n, \sum_{i=1}^n f_i=0\right\}
\]
is $n-1$. Indeed, let $W=\{ f\in \Rea^n:\, \sum_{i=1}^n f_i=0\}$. Clearly $\dim(W)=n-1$. We have $U=\Phi(W)$, where $\Phi\in \Rea^{|E|\times n}$ is defined by
\[
    \Phi(e,i)=
    \begin{cases}
        \ell_{i j} & \text{ if } e=\{i,j\} \text{ for some } j\in[n],\\
        0 & \text{ otherwise.}
        \end{cases}
\]
Let $g\in \Rea^n$ such that $\Phi g=0$. Let $i\in[n]$. Since the image of $p$ is of size at least $3$, there exist $j,k\in[n]$ such that $p(i),p(j),p(k)$ are pairwise distinct. Then, we have
\begin{align*}
    0&=(\Phi g)_{\{i,j\}}= (g_i+g_j)\ell_{i j},
\\
    0&=(\Phi g)_{\{i,k\}}= (g_i+g_k)\ell_{i k},
\\
    0&=(\Phi g)_{\{j,k\}}= (g_j+g_k)\ell_{j k}.
\end{align*}
We obtain $g_i=-g_j=g_k=-g_i$. 
That is, $g_i=0$. Therefore, $g=0$. Thus, $\Phi$ has linearly independent columns, and so $\dim(U)=\dim(\Phi(W))= \dim(W)=n-1$, as wanted.

Hence, $n/2$ has multiplicity at least $n-1$ as an eigenvalue of $L^{-}(K_n,p)$ (and thus also as an eigenvalue of $L(K_n,p)$).
\end{proof}

We conjecture that for the mappings considered in Proposition \ref{prop:large_eigenvalues_balanced_p}, $n/2$ is the second largest eigenvalue:

\begin{conjecture}\label{conj:largest_eig_spherical}
Let $d\geq 3$, and let $p:[n]\to\Rea^d$ such that $\|p(i)\|=1$ for all $i\in[n]$ and $\sum_{i=1}^n p(i)=0$. Assume that the image of $p$ is of size at least $3$. Then, the second largest eigenvalue of $L(K_n,p)$ is $n/2$, and its multiplicity is exactly $n-1$.
\end{conjecture}

In the case $d=2$ we can say more:

\begin{proposition}\label{prop:2d}
Let $p:[n]\to\Rea^2$ such that $\|p(i)\|=1$ for all $i\in[n]$ and $\sum_{i=1}^n p(i)=0$. Assume that the $p$ is injective. Then, the spectrum of $L(K_n,p)$ is
\[
\left\{ 0^{(3)}, \frac{n}{2}^{(2n-4)},n^{(1)}\right\}.
\]
\end{proposition}
Proposition \ref{prop:2d} extends Zhu's result \cite[Remark 3.4.1]{zhu2013quantitative}, which states that the same conclusion holds under the more restrictive assumption that $p$ maps $[n]$ to the roots of unity of order $n$. In particular, this shows that there are infinitely many embeddings $p:[n]\to \Rea^2$ attaining the supremum $a_2(K_n)=n/2$. For completeness, we include a proof, following the proof of the special case sketched by Zhu in \cite{zhu2013quantitative}.
\begin{proof}
Let $x,y\in \Rea^{2n}$ be defined by 
\[
    x= (1,0,1,0,\ldots,1,0)^t,\ \ 
    y= (0,1,0,1,\ldots,0,1)^t.
\]
For $i\in [n]$, let $p_x(i),p_y(i)$ be the two coordinates of $p(i)$, and let
\[
    p^{\perp}(i)= (-p_y(i),p_x(i))^t.
\]
Define $r\in \Rea^{2n}$ by
\begin{align*}
    r&=
    (p_x^{\perp}(1),p_y^{\perp}(1),p_x^{\perp}(2),p_y^{\perp}(2),\ldots,p_x^{\perp}(n),p_y^{\perp}(n))^t
    \\&=
    (-p_y(1),p_x(1),-p_y(2),p_x(2),\ldots,-p_y(n),p_x(n))^t.
\end{align*}
It is easy to check (using the fact that $\sum_{i=1}^n p(i)=0$) that $x,y,r$ belong to the kernel of $L(K_n,p)$, and moreover, form an orthogonal set.

We identify the mapping $p$ with the vector
\[
    p= (p_x(1),p_y(1),\ldots,p_x(n),p_y(n))^t\in \Rea^{2n}.
\]
By Lemma \ref{lemma:n_eigenvalue}, $p$ is an eigenvector of $L(K_n,p)$ with eigenvalue $n$.

We will show that
\begin{equation}\label{eq:zhu}
    L(K_n,p)= \frac{n}{2}I +\frac{1}{2} (p p^t -x x^t -y y^t -r r^t).
\end{equation}
Then, it immediately follows that every vector in $\Rea^{2d}$ orthogonal to $p,x,y$ and $r$ is an eigenvector of $L(K_n,p)$ with eigenvalue $n/2$, as wanted.

We can write $L(K_n,p)$ as a $n\times n$ block matrix (see \cite[Section 4.4]{jordan2020rigidity}), formed by $2\times 2$ blocks
\[
[L(K_n,p)]_{i,j}= -d_{i j} d_{i j}^t,
\]
for $i\neq j\in [n]$, and
\[
[L(K_n,p)]_{i,i}=\sum_{j\in[n]\setminus\{i\}} d_{i j} d_{i j}^t
\]
for $i\in[n]$.
It is then easy to check that proving \eqref{eq:zhu} is equivalent to showing that, for all $i\in[n]$, 
\begin{equation}\label{eq:block_zhu_diag}
 \sum_{j\in[n]\setminus\{i\}} d_{i j} d_{i j}^t=\frac{1}{2} p(i)p(i)^t- \frac{1}{2}p^{\perp}(i) (p^{\perp}(i))^t  + \frac{n-1}{2} I, 
\end{equation}
and for all $i\neq j\in[n]$
\begin{equation}\label{eq:block_zhu}
-d_{i j} d_{i j}^t= \frac{1}{2} p(i)p(j)^t- \frac{1}{2}p^{\perp}(i) (p^{\perp}(j))^t  - \frac{1}{2} I.
\end{equation}
First, note that \eqref{eq:block_zhu_diag} follows from \eqref{eq:block_zhu}. Indeed, let $i\in[n]$. By \eqref{eq:block_zhu}, and using the fact that $\sum_{j\in[n]\setminus\{i\}} p(j)=-p(i)$ (and similarly, $\sum_{j\in[n]\setminus\{i\}} p^{\perp}(j)=-p^{\perp}(i)$), we obtain
\begin{multline*}
\sum_{j\in[n]\setminus\{i\}} d_{i j} d_{i j}^t= \sum_{j\in[n]\setminus\{i\}}\left( \frac{1}{2} I -\frac{1}{2}p(i)p(j)^t+\frac{1}{2}p^{\perp}(i)(p^{\perp}(j))^t\right)
\\
=\frac{n-1}{2}I -\frac{1}{2}p(i)\left(\sum_{j\in[n]\setminus\{i\}}p(j)^t\right)
+\frac{1}{2}p^{\perp}(i)\left(\sum_{j\in[n]\setminus\{i\}}p^{\perp}(j)^t\right)
\\
=\frac{n-1}{2}I +\frac{1}{2}p(i)p(i)^t
-\frac{1}{2}p^{\perp}(i)p^{\perp}(i)^t.
\end{multline*}
So, we are left to show that \eqref{eq:block_zhu} holds for all $i\neq j\in[n]$.

Let $i\neq j\in[n]$, and let
\[
    M= d_{i j} d_{i j}^t+ \frac{1}{2} p(i)p(j)^t- \frac{1}{2}p^{\perp}(i) (p^{\perp}(j))^t  - \frac{1}{2} I.
\]
We will show that $M=0$. We denote the four coordinates of $M$ by $M_{x x}$, $M_{x y}$, $M_{y x}$ and $M_{y y}$. Since $\|p(i)\|=\|p(j)\|=1$, we have
\[
    \|p(i)-p(j)\|^2 = 2(1-p(i)\cdot p(j)),
\]
and therefore
\[
d_{i j} d_{i j}^t= \frac{(p(i)-p(j))(p(i)-p(j))^t}{2(1-p(i)\cdot p(j))}.
\]
Using this and the fact that $p_x(i)^2+p_y(i)^2=p_x(j)^2+p_y(j)^2=1$, we obtain
\begin{align*}
&M_{x x}= \frac{(p_x(i)-p_x(j))(p_x(i)-p_x(j))}{2(1-p_x(i)p_x(j)-p_y(i)p_y(j))} +\frac{1}{2}p_x(i)p_x(j)-\frac{1}{2}p_y(i)p_y(j)-\frac{1}{2}
\\&= \frac{1}{2(1-p_x(i)p_x(j)-p_y(i)p_y(j))}\big( p_x(i)^2-2p_x(i)p_x(j)+p_x(j)^2 +p_x(i)p_x(j)\\&-p_y(i)p_y(j)
+p_y(i)^2p_y(j)^2-p_x(i)^2 p_x(j)^2 -1 +p_x(i)p_x(j)+p_y(i)p_y(j)
\big)
\\
&=\frac{p_x(i)^2 + p_x(j)^2 -p_x(i)^2p_x(j)^2-1 +(1-p_x(i)^2)(1-p_x(j)^2)}{2(1-p_x(i)p_x(j)-p_y(i)p_y(j))}=0.
\end{align*}
Similarly, $M_{y y}=0$. Finally, we have
\begin{align*}
M_{x y}&= \frac{(p_x(i)-p_x(j))(p_y(i)-p_y(j))}{2(1-p_x(i)p_x(j)-p_y(i)p_y(j))} +\frac{1}{2}p_x(i)p_y(j)+\frac{1}{2}p_y(i)p_x(j)
\\&= 
\frac{1}
{2(1-p_x(i)p_x(j)-p_y(i)p_y(j))}
\big(
p_x(i)p_y(i)-p_y(i)p_x(j)-p_x(i)p_y(j)\\&+p_x(j)p_y(j)
+p_x(i)p_y(j)+p_y(i)p_x(j)
-p_x(i)^2 p_x(j)p_y(j)\\
&-p_x(i)p_y(i)p_x(j)^2  -p_x(i)p_y(i)p_y(j)^2 -p_y(i)^2p_x(j)p_y(j)
\big)
\\
&=\frac{p_x(i)p_y(i)+p_x(j)p_y(j)
-p_x(j)p_y(j)
-p_x(i)p_y(i)}
{2(1-p_x(i)p_x(j)-p_y(i)p_y(j))}
=0,
\end{align*}
and similarly $M_{y x}=0$. So, $M=0$ as wanted.
\end{proof}

The following very simple Lemma will be useful for finding the additional eigenvectors of $L^{-}(T(n,d+1),q^{\triangle})$ and $L^{-}(T(n,2d),q^{\diamond})$.

\begin{lemma}\label{lemma:condition_for_eigenvector}
Let $G=(V,E)$ and $p:V\to \Rea^d$ such that $p(i)\neq p(j)$ for all $\{i,j\}\in E$. Let $L^{-}=L^{-}(G,p)$ and $\phi\in \Rea^{E}$. Let $\text{supp}(\phi)\subset E$ be the support of $\phi$. Assume that the following conditions hold:
\begin{enumerate}
    \item For all $e\in E\setminus \supp(\phi)$, $
        \cos(\theta(e,e'))=\cos(\theta(e,e''))
    $
    for every $e',e''\in \supp(\phi)$ such that $|e\cap e'|=|e\cap e''|=1$.
    \item For all $v\in V$,
   $
    \sum_{e\in E:\, v\in e} \phi(e)=0.
    $
    \item There is $\lambda\in \Rea$ such that for all $e\in\supp(\phi)$
    \[
        \sum_{e'\in E:\, |e\cap e'|=1} \cos(\theta(e,e'))\phi(e')=(\lambda-2)\phi(e).
    \]
\end{enumerate}
Then, $\phi$ is an eigenvector of $L^{-}$ with eigenvalue $\lambda$.
\end{lemma}
\begin{proof}
By Lemma \ref{lemma:down_laplacian}, conditions $(1)$ and $(2)$ imply that $L^{-}\phi(e)=0=\phi(e)$ for every $e\in E\setminus \supp(\phi)$, and condition $(3)$ says exactly that $L^{-}\phi(e)=\lambda\phi(e)$ for all $e\in\supp(\phi)$. Therefore, we obtain $L^{-}\phi=\lambda\phi$, as wanted.
\end{proof}

\subsection{The spectrum of $L(T(n,d+1),q^{\triangle})$}

\simplexrepeated*

Note that for $n=d+1$ the proof below gives a second proof of Theorem \ref{thm:simplex_spectrum}.
\begin{proof}

Denote $T(n,d+1)=(V,E)$. First note that, as for all frameworks, $0$ is an eigenvalue of $L(T(n,d+1),q^{\triangle})$ with multiplicity at least $\binom{d+1}{2}=d(d+1)/2$.

Moreover, note that $L(T(n,d+1),q^{\triangle})=L(K_n,q^{\triangle})$. Thus, by Lemma \ref{lemma:n_eigenvalue} , $n$ is an eigenvalue of $L(T(n,d+1),q^{\triangle})$.
Also, since $\|q^{\triangle}(v)\|=1$ for all $v\in V$ and $\sum_{v\in V} q^{\triangle}(v)=0$, then by Proposition \ref{prop:large_eigenvalues_balanced_p}, $n/2$ is an eigenvalue of $L(T(n,d+1),q^{\triangle})$ with multiplicity at least $n-1$.

Therefore, we are left to show that $n/(d+1)$ is an eigenvalue with multiplicity at least $(d-2)(d+1)/2$ and that $n/(2(d+1))$ is an eigenvalue with multiplicity at least $(n-d-1)(d-1)$.

Since the non-zero eigenvalues of $L(T(n,d+1),q^{\triangle})$ and $L^{-}=L^{-}(T(n,d+1),q^{\triangle})$ are the same, we will find the corresponding eigenvectors for $L^{-}$.

Denote by $V_1,\ldots,V_{d+1}$ the sides of $T(n,d+1)$. For every $i\neq j\in[d+1]$, let $E(V_i,V_j)$ be the set of edges with one endpoint in $V_i$ and the other in $V_j$.
Let $i,j,k\in [d+1]$ such that $i\neq j,k$. Let $v\in V_i$, $u\in V_j$ and $w\in V_k$. Denote by $\theta_{uvw}$ the angle between $q^{\triangle}(v)-q^{\triangle}(u)$ and $q^{\triangle}(v)-q^{\triangle}(w)$. Then, we have
\begin{equation}\label{eq:anglesimplex}
    \cos(\theta_{uvw})=\begin{cases}
      \frac{1}{2} & \text{ if } j\neq k,\\  
      1 & \text{ if } j=k.
    \end{cases}
\end{equation}

\textbf{Eigenvalue $\bm{n/(d+1)}$:}  Assume $d\geq 3$. Let $i_1,i_2,i_3,i_4\in[d+1]$ be four distinct integers. 
Define $ \Phi_{i_1,i_2,i_3,i_4}\in \Rea^{E}$ by
\[
 \Phi_{i_1,i_2,i_3,i_4}(e)=\begin{cases}
     1 & \text{ if } e\in E(V_{i_1},V_{i_2})\cup E(V_{i_3},V_{i_4}),\\
     -1 & \text{ if } e\in E(V_{i_2},V_{i_3})\cup E(V_{i_1},V_{i_4}),\\
     0 & \text{ otherwise.}
 \end{cases}
\]
Note that for every $e\notin \supp(\Phi_{i_1,i_2,i_3,i_4})$ and every $e'\in \supp(\Phi_{i_1,i_2,i_3,i_4})$ such that $|e\cap e'|=1$, $\cos(\theta(e,e'))=1/2$, and that for all $v\in V$, we have $\sum_{e\in E:\, v\in e} \Phi_{i_1,i_2,i_3,i_4}(e)=0$.

Moreover, let $e=\{u,v\}\in\supp(\Phi_{i_1,i_2,i_3,i_4})$. 
Assume $u\in V_{i_1}$ and $v\in V_{i_2}$ (the other cases are analyzed similarly). Then, using \eqref{eq:anglesimplex}, we obtain
\begin{multline*}
        \sum_{e'\in E:\, |e\cap e'|=1} \cos(\theta(e,e'))\Phi_{i_1,i_2,i_3,i_4}(e')=\sum_{w\in V_{i_1}\setminus\{u\}} 1 + \sum_{w\in V_{i_2}\setminus\{v\}} 1 \\-\sum_{w\in V_{i_3}}\frac{1}{2}-\sum_{w\in V_{i_4}}\frac{1}{2}= \frac{n}{d+1}-2= \left(\frac{n}{d+1}-2\right)\Phi_{i_1,i_2,i_3,i_4}(e).
\end{multline*}
 
Therefore, by Lemma \ref{lemma:condition_for_eigenvector}, $\Phi_{i_1,i_2,i_3,i_4}$ is an eigenvector of $L^{-}$ with eigenvalue $n/(d+1)$.

Let 
\begin{multline*}
    I=  \left\{ (1,2,3,k),(1,3,2,k) :\, k\in [d+1]\setminus\{1,2,3\}\right\}
    \\
    \cup \left\{ (2,3,j,k)  :\, j,k\in[d+1]\setminus\{1,2,3\},j<k \right\},
\end{multline*}
and let
\[
    B= \left\{ \Phi_{i_1,i_2,i_3,i_4}\right\}_{(i_1,i_2,i_3,i_4)\in I}.
\]
Note that
\[
    |B|= |I|= 2(d-2)+ \binom{d-2}{2} = \frac{(d-2)(d+1)}{2}.
\]
We will show that $B$ is a linearly independent set:

For each $(i_1,i_2,i_3,i_4)\in I$, let $\alpha_{i_1,i_2,i_3,i_4}\in \Rea$. Assume that
\[
    \sum_{(i_1,i_2,i_3,i_4)\in I} \alpha_{i_1,i_2,i_3,i_4} \Phi_{i_1,i_2,i_3,i_4}=0.
\]
Let $j,k\in[d+1]\setminus\{1,2,3\}$ such that $j<k$. Note that for every $u\in V_j$ and $w\in V_k$, $\Phi_{2,3,j,k}$ is the only vector in $B$ containing $\{u,w\}$ in its support. Therefore, we must have $\alpha_{2,3,j,k}=0$.

Now, let $k\in [d+1]\setminus\{1,2,3\}$. For every $u\in V_1$, $w\in V_k$, the only vectors in $B$ containing $\{u,w\}$ in their support are $\Phi_{1,2,3,k}$ and $\Phi_{1,3,2,k}$. Therefore, we must have 
$\alpha_{1,2,3,k}=-\alpha_{1,3,2,k}$.
Hence, we obtain a linear relation
\[
    \sum_{k\in[d+1]\setminus\{1,2,3\}} \alpha_{1,2,3,k} (\Phi_{1,2,3,k}-\Phi_{1,3,2,k})=0.
\]
Note that $\Phi_{1,2,3,k}-\Phi_{1,3,2,k}=\Phi_{1,2,k,3}$. So, we obtain
\[
    \sum_{k\in[d+1]\setminus\{1,2,3\}} \alpha_{1,2,3,k} \Phi_{1,2,k,3}=0.
\]
But note that each vector $\Phi_{1,2,k,3}$ contains a unique edge in its support (for example, any edge $\{u,w\}$ where $u\in V_2$ and $w\in V_k$). Therefore, $\{\Phi_{1,2,k,3}\}_{k\in[d+1]\setminus\{1,2,3\}}$ are independent, and hence $\alpha_{1,2,3,k}=\alpha_{1,3,2,k}=0$ for all $k\in[d+1]\setminus\{1,2,3\}$. Thus, $B$ is linearly independent.
So, $n/(d+1)$ is an eigenvalue of $L^{-}$ with multiplicity at least $(d-2)(d+1)/2$.

\textbf{Eigenvalue $\bm{n/(2(d+1))}$:} 

Let $i,j,k\in[d+1]$ be three distinct 
indices, and let $u,v\in V_i$. Define
\[
    \Psi_{u,v,j,k}= \sum_{w\in V_j} (1_{\{u,w\}}-1_{\{v,w\}})+\sum_{w\in V_k} (1_{\{v,w\}}-1_{\{u,w\}}).
\]
Let $e=\{x,y\}\notin \supp(\Psi_{u,v,j,k})$. Assume $x\in V_r$ and $y\in V_t$. If $\{r,t\}=\{i,j\}$ or $\{r,t\}=\{i,k\}$, then $\cos(\theta(e,e'))=1$ for all $e'\in \supp(\Psi_{u,v,j,k})$ such that $|e\cap e'|=1$. Otherwise, $\cos(\theta(e,e'))=1/2$ for all $e'\in \supp(\Psi_{u,v,j,k})$  such that $|e\cap e'|=1$.

Also, it is easy to check that for every $w\in V$, $\sum_{e\in E,\\ w\in e} \Psi_{u,v,j,k}(e)=0$.

Finally, let $e=\{x,y\}\in\supp(\Psi_{u,v,j,k})$. Assume $x=u$ and $y\in V_j$ (the other cases are similar). Then, by \eqref{eq:anglesimplex}, we have
\begin{align*}
        &\sum_{e'\in E:\, |e\cap e'|=1} \cos(\theta(e,e'))\Psi_{u,v,j,k}(e')
        \\&=\sum_{w\in V_j\setminus\{y\}} \cos(\theta_{y u w}) - \sum_{w\in V_k}\cos(\theta_{y u w}) - \cos(\theta_{u y v})
        \\
        &=\sum_{w\in V_j\setminus\{y\}} 1  -\sum_{w\in V_k} \frac{1}{2} -1 = \frac{n}{2(d+1)}-2=\left(\frac{n}{2(d+1)}-2\right)\Psi_{u,v,j,k}(e).
        \end{align*}
Therefore, by Lemma \ref{lemma:condition_for_eigenvector}, $\Psi_{u,v,j,k}$ is an eigenvector of $L^{-}$ with eigenvalue $n/(2(d+1))$.

For every $i\in[d+1]$, fix some $j(i)\in [d+1]\setminus\{i\}$ and $u_i\in V_i$, and let
\[
    J_i=\left\{(u_i,v,j(i),k):\, v\in V_i\setminus\{u_i\}, \, k\in[d+1]\setminus\{i,j(i)\}\right\}
\]
and
\[
    B_i= \left\{ \Psi_{u,v,j,k}\right\}_{(u,v,j,k)\in J_i}.
\]
Let $B=\cup_{i=1}^{d+1} B_i$. Note that
\[
    |B|=\sum_{i=1}^{d+1}|J_i|= (d+1)\left(\frac{n}{d+1}-1\right)(d-1)  =    (n-d-1)(d-1).
\]
We will show that $B$ is a linearly independent set. Let $i\neq i'\in[d+1]$ and let $(u_i,v,j,k)\in J_i$ and $(u_{i'},v',j',k')\in J_{i'}$. We will show that $\Psi_{u_i,v,j,k}$ and $\Psi_{u_{i'},v',j',k'}$ are orthogonal. Indeed, the supports of the two vectors are disjoint unless $i'\in \{j,k\}$ and $i\in \{j',k'\}$. In this case, the intersection of the two supports consists of the four edges $\{u_i,u_{i'}\},\{u_i,v'\},\{v,u_{i'}\},\{v,v'\}$, and it is easy to check that 
we have $\Psi_{u_i,v,j,k}\cdot\Psi_{u_{i'},v',j',k'}=0$.

Therefore, it is enough to check that for each $i\in[d+1]$, $B_i$ is linearly independent. But this follows from the fact that for every $(u_i,v,j(i),k)\in J_i$, there is a unique edge in the support of $\Psi_{u_i,v,j(i),k}$ (any edge of the form $\{v,w\}$ where $w\in V_k$). So $B$ is linearly independent, and therefore $n/(2(d+1))$ is an eigenvector of $L^{-}$ with multiplicity at least $(n-d-1)(d-1)$. 
\end{proof}

\begin{remark}
In \cite{jordan2020rigidity}, a lower bound $\lambda_{\binom{d+1}{2}+1}(L(T(n,d+1),q^{\triangle}))\geq \frac{2d^2+d}{2(d+1)^3}n$ (for $n$ divisible by $d+1$) was stated, in contradiction to Theorem \ref{thm:simplex_repeated}. After fixing a typo in \cite[Lemma 5.4]{jordan2020rigidity}, the corrected lower bound obtained in \cite{jordan2020rigidity} is
\[
\lambda_{\binom{d+1}{2}+1}(L(T(n,d+1),q^{\triangle}))\geq \frac{d n}{2(d+1)^2}.
\]
From Theorem \ref{thm:simplex_repeated} it follows that in fact, for $n$ divisible by $d+1$, 
\[
\lambda_{\binom{d+1}{2}+1}(L(T(n,d+1),q^{\triangle}))= \frac{n}{2(d+1)}.
\]
\end{remark}

\subsection{The spectrum of $L(T(n,2d),q^{\diamond})$}

\crosspolytopespectrum*
\begin{proof}

Denote $T(n,2d)=(V,E)$.
First, note that, as for all frameworks, $0$ is an eigenvalue of $L(T(n,2d),q^{\diamond})$ with multiplicity at least $d(d+1)/2$.

Note that $L(T(n,2d),q^{\diamond})=L(K_n,q^{\diamond})$. Thus, by Lemma \ref{lemma:n_eigenvalue} , $n$ is an eigenvalue of $L(T(n,2d),q^{\diamond})$.
Also, since $\|q^{\diamond}(v)\|=1$ for all $v\in V$ and $\sum_{v\in V} q^{\diamond}(v)=0$, then by Proposition \ref{prop:large_eigenvalues_balanced_p}, $n/2$ is an eigenvalue of $L(T(n,2d),q^{\diamond})$ with multiplicity at least $n-1$.

Therefore, we are left to show that $n/d$ is an eigenvalue with multiplicity at least $d(d-1)/2$ and that $n/(2d)$ is an eigenvalue with multiplicity at least $n(d-1)-d^2$.

Since the non-zero eigenvalues of $L(T(n,2d),q^{\diamond})$ and $L^{-}=L^{-}(T(n,2d),q^{\diamond})$ are the same, we will find the corresponding eigenvectors for $L^{-}$.

Denote the parts of $T(n,2d)$ by $V_{1,1},V_{1,-1},V_{2,1},V_{2,-1}\ldots,V_{d,1},V_{d,-1}$ where, for $i\in[d]$ and $s\in\{1,-1\}$, $q^{\diamond}(V_{i,s})$ is the singleton set consisting of the vector $s e_i$, where $e_i$ is the $i$-th standard vector in $\Rea^d$. 
For every $i,j\in [d]$ and $x,y\in\{1,-1\}$, let $E(V_{i,x},V_{j,y})$ be the set of edges with one endpoint in $V_{i,x}$ and the other in $V_{j,y}$.

Let $u\in V_{i,x}$ and $v\in V_{j,y}$, where $i\neq j\in[d]$ and $x,y\in\{1,-1\}$. Let $w\neq u,v$, and assume $w\in V_{k,z}$ for $k\in [d]$ and $z\in\{1,-1\}$.

Let $\theta_{v u w}$ be the angle between $q^{\diamond}(u)-q^{\diamond}(v)$ and $q^{\diamond}(u)-q^{\diamond}(w)$. Then,
\begin{equation}\label{eq:angles2d}
\cos(\theta_{v u w})=\begin{cases}
    \frac{1}{2} & \text{ if } k\neq i,j,\\
    \frac{1}{\sqrt{2}} & \text{ if } k=i, z=-x,\\
    1 & \text{ if } k=j, z=y,\\
    0 & \text{ otherwise.}
    \end{cases}
\end{equation}

\textbf{Eigenvalue $\bm{n/d}$:}  Let $i< j \in [d]$.
Let $\phi_{ij}\in \Rea^E$ be defined by
\[
\phi_{ij}(e)= \begin{cases}
            1 & \text{ if } e\in E(V_{i,1},V_{j,1})\cup E(V_{i,-1},V_{j,-1}),\\
            -1 & \text{ if } e\in E(V_{i,-1},V_{j,1})\cup E(V_{i,1},V_{j,-1}),\\
            0 & \text{ otherwise}.
\end{cases}
\]
Let $e\notin \supp(\phi_{i j})$. 
If $e\in E(V_{i,1},V_{i,-1})\cup E(V_{j,1},V_{j,-1})$, then $\cos(\theta(e,e'))=1/\sqrt{2}$
for every $e'\in \supp(\phi_{i j})$ such that $|e\cap e'|=1$.
If $e\in E(V_{k,1},V_{k,-1})$ for some $k\in[d]\setminus\{i,j\}$, then there are no edges $e'\in \supp(\phi_{i j})$ such that $|e\cap e'|=1$.
If $e\notin \cup_{k=1}^d E(V_{k,1},V_{k,-1})$, then $\cos(\theta(e,e'))=1/2$
for every $e'\in \supp(\phi_{i j})$ such that $|e\cap e'|=1$.

Also, note that for all $v\in V$, we have $\sum_{e\in E:\, v\in e} \phi_{i j}(e)=0$.

Moreover, let $e=\{u,v\}\in\supp(\phi_{i j})$. Assume $u\in V_{i, 1}$ and $v\in V_{j, 1}$ (the other cases are analyzed similarly). Then, using \eqref{eq:angles2d}, we obtain
\begin{align*}
        \sum_{e'\in E:\, |e\cap e'|=1} \cos(\theta(e,e'))\phi_{ i j}(e') &=\sum_{w\in V_{i,1}\setminus\{u\}} 1 + \sum_{w\in V_{j,1}\setminus\{v\}} 1 
      \\
      &= \frac{n}{d}-2 
      = \left(\frac{n}{d}-2\right)\phi_{i j}(e).
\end{align*}
So, by Lemma \ref{lemma:condition_for_eigenvector}, $\phi_{ij}$ is an eigenvector of $L^{-}$ with eigenvalue $n/d$.
Since the supports of the vectors $\{\phi_{ij}\}_{i<j\in[d]}$ are pairwise disjoint, they form a linearly independent set. Thus, $n/d$ is an eigenvalue of $L^{-}$ with multiplicity at least $d(d-1)/2$.

\textbf{Eigenvalue $\bm{n/(2d)}$:} 
Let $i\neq j\in[d]$. Define $\psi_{i j}\in \Rea^E$ by
\[
    \psi_{ij}(e)= \begin{cases}
        1 & \text{ if } e\in E(V_{i,1},V_{j,1})\cup E(V_{i,-1},V_{j,1}),\\
        -1 & \text{ if } e\in E(V_{i,1},V_{j,-1})\cup E(V_{i,-1},V_{j,-1}),\\
        0 & \text{ otherwise.}
    \end{cases}
\]
For $k\neq i,j$, let $   \Psi_{i j k}= \psi_{i j}-\psi_{k j}$.

Let $e\notin \supp(\Psi_{i j k})$. 
If $e\in \cup_{t=1}^d E(V_{t,1},V_{t,-1})$, then $\cos(\theta(e,e'))=1/\sqrt{2}$
for every $e'\in \supp(\Psi_{i j k})$ such that $|e\cap e'|=1$ (note that if $t\neq i,j,k$ then there are no such edges $e'$).
If $e\notin \cup_{t=1}^d E(V_{t,1},V_{t,-1})$, then $\cos(\theta(e,e'))=1/2$
for every $e'\in \supp(\Psi_{i j k})$ such that $|e\cap e'|=1$.

Also, note that for all $v\in V$, we have $\sum_{e\in E:\, v\in e} \Psi_{i j k}(e)=0$.

Moreover, let $e=\{u,v\}\in\supp(\Psi_{i j k})$. Assume $u\in V_{i, 1}$ and $v\in V_{j, 1}$ (the other cases are analyzed similarly). Then, using \eqref{eq:angles2d}, we obtain
\begin{align*}
        \sum_{e'\in E:\, |e\cap e'|=1} \cos(\theta(e,e'))\Psi_{ i j k}(e') &=\sum_{w\in V_{i,1}\setminus\{u\}} 1 + \sum_{w\in V_{j,1}\setminus\{v\}} 1
        - \sum_{w\in V_{k,1}\cup V_{k,-1}} \frac{1}{2}
      \\
      &= \frac{n}{2d}-2
      = \left(\frac{n}{2d}-2\right)\Psi_{i j k}(e).
\end{align*}

Therefore, by Lemma \ref{lemma:condition_for_eigenvector}, $\Psi_{i j k}$ is an eigenvector of $L^{-}$ with eigenvalue $n/(2d)$. 

For each $j\in [d]$, fix $i(j)\in[d]\setminus\{j\}$. Let
\[
    I_j=\{(i(j),j,k):\, k\in [d]\setminus\{j,i(j)\}\}
\]
and $I=\cup_{j\in[d]} I_j$. Note that $|I|=d(d-2)$. We will show that $\{\Psi_{ijk}\}_{(i,j,k)\in I}$ is a linearly independent set.

First, note that the vectors $\{\psi_{ij}\}_{i\neq j\in[d]}$ form an orthogonal set. Indeed, let $i,j,k,s\in[d]$ such that $i\neq j$ and $k\neq s$. If $\{i,j\}\neq \{k,s\}$ then $\psi_{i j}$ and $\psi_{k s}$ have disjoint supports, and therefore $\psi_{ij}\cdot \psi_{k s}=0$. Otherwise, if $i=s$ and $j=k$
we obtain
\begin{multline*}
\psi_{ij}\cdot \psi_{ji}
= \sum_{e\in E(V_{i,1},V_{j,1})} 1\cdot 1 + \sum_{e\in E(V_{i,1},V_{j,-1})} (-1)\cdot 1 \\ + \sum_{e\in E(V_{i,-1},V_{j,1})} 1\cdot (-1) + \sum_{e\in E(V_{i,-1},V_{j,-1})} (-1)\cdot(-1)=0.
\end{multline*}
Hence, for $\{i,j,k\}\in\binom{[d]}{3}$ and $\{i',j',k'\}\in\binom{[d]}{3}$ such that $j\neq j'$, we have $\Psi_{i j k}\cdot \Psi_{i' j' k'}=0$.
Therefore, we are left to show that for all $j\in[d]$, $\{\Psi_{i j k}\}_{(i,j,k)\in I_j}$ is linearly independent. But this follows from the fact that for every $(i,j,k)\in I_j$ there is an edge in the support of $\Psi_{i j k}$ that is unique to $\Psi_{i j k}$ (for example, any edge $\{u,v\}$ where $u\in V_{j,1}$ and $v\in V_{k,1}$).

We now complete the set $\{\Psi_{i j k}\}_{(i,j,k)\in I}$ to an eigenbasis of $n/2d$.
Let $i\neq j \in[d]$ and $x\in\{1,-1\}$. Let $u\neq v\in V_{i,x}$.

Define $f_{j}^{u,v}\in \Rea^E$ by
\[
    f_{j}^{u,v}= \sum_{w\in V_{j,1}} (1_{\{u,w\}}-1_{\{v,w\}})+\sum_{w\in V_{j,-1}} (-1_{\{u,w\}}+1_{\{v,w\}}).
\]
Let $e\notin \supp(f_j^{u,v})$.
If $e\in E(V_{i,x},V_{j,1})\cup E(V_{i,x},V_{j,-1})$, then $\cos(\theta(e,e'))=1$ for all $e'\in \supp(f_j^{u,v})$ such that $|e\cap e'|=1$. If $e\in E(V_{i,1},V_{i,-1})\cup E(V_{j,1},V_{j,-1})$, then $\cos(\theta(e,e'))=1/\sqrt{2}$ for all $e'\in \supp(f_j^{u,v})$ such that $|e\cap e'|=1$. If $e\in E(V_{i,-x},V_{j,1})\cup E(V_{i,-x},V_{j,-1})$, then $\cos(\theta(e,e'))=0$ for all $e'\in \supp(f_j^{u,v})$ such that $|e\cap e'|=1$. Otherwise, $\cos(\theta(e,e'))=1/2$ for all $e'\in \supp(f_j^{u,v})$ such that $|e\cap e'|=1$.

In addition, it is easy to check that for every $w\in V$, $\sum_{e\in E,\\ w\in e} f^{u,v}_j(e)=0$.

Finally, let $e=\{a,b\}\in\supp(f_j^{u,v})$. Assume $a=u$ and $b\in V_{j,1}$ (the other cases are similar). Then, by \eqref{eq:angles2d}, we have
\begin{align*}
        &\sum_{e'\in E:\, |e\cap e'|=1} \cos(\theta(e,e')) f_j^{u,v}(e')
        \\&=\sum_{w\in V_{j,1}\setminus\{b\}} \cos(\theta_{b u w}) - \sum_{w\in V_{j,-1}}\cos(\theta_{b u w}) - \cos(\theta_{u b v})
        \\
        &=\sum_{w\in V_j\setminus\{b\}} 1  -\sum_{w\in V_k} 0 -1 = \frac{n}{2d}-2=\left(\frac{n}{2d}-2\right)f_j^{u,v}(e).
        \end{align*}
Therefore, by Lemma \ref{lemma:condition_for_eigenvector}, $f_j^{u,v}$ is an eigenvector of $L^{-}$ with eigenvalue $n/(2d)$.

For each $i\in [d]$ and $x\in\{1,-1\}$, fix some $u(i,x)\in V_{i,x}$. 
For $j\in[d]\setminus\{i\}$, let
\[
J_{i,x,j}=\{ (u(i,x),v,j):\, u(i,x)\neq v\in V_{i,x}\}.
\]
and let $J=\cup_{i\in[d]}\cup_{x\in\{1,-1\}}\cup_{j\in [d]\setminus \{i\}} J_{i,x,j}$. Note that $|J|=2d(d-1)(n/(2d)-1)$.

We will show that $\{f_j^{u,v}\}_{(u,v,j)\in J}$ are linearly independent.

Let $u\neq v\in V_{i,x}$ and $u'\neq v'\in V_{i',x'}$. Let $j\neq i$ and $j'\neq i'$.
Assume $(i,x,j)\neq (i',x',j')$.
We will show that $f_{j}^{u,v}$ and $f_{j'}^{u',v'}$ are orthogonal. Indeed, the supports of the two vectors are disjoint unless $i'=j$ and $j'=i$. But it is easy to check that also in this case we have $f_{j}^{u,v}\cdot f_{j'}^{u',v'}=0$.

Therefore, it is enough to show that for every $i\neq j\in[d]$ and $x\in\{1,-1\}$, $\{f_j^{u,v}\}_{(u,v,j)\in J_{i,x,j}}$ is linearly independent. But again, this follows from the fact that for every $(u,v,j)\in J_{i,x,j}$ there is an edge in the support of $f_{j}^{u,v}$ that is unique to it (for example, the edge $\{v,w\}$ for any $w\in V_{j,1}$).

Finally, note that $\psi_{ij}\cdot f_{k}^{u,v}=0$ for every $i\neq j\in[d]$ and $(u,v,k)\in J$. Therefore, $\Psi_{i j k}\cdot f_{m}^{u,v}=0$ for all $(i,j,k)\in I$ and $(u,v,m)\in J$. Thus, the eigenvectors $\{\Psi_{i j k}\}_{(i,j,k)\in I}\cup \{f_j^{u,v}\}_{(u,v,j)\in J}$ form a linearly independent set.

So, $n/(2d)$ is an eigenvalue of $L^{-}$ with multiplicity at least
\[
    d(d-2)+2d(d-1)(n/(2d)-1)= n(d-1)-d^2,
\]
as wanted.
\end{proof}

It was shown in \cite[Lemma 4.5]{jordan2020rigidity} that the removal of a vertex reduces the $d$-dimensional algebraic connectivity of a graph by at most $1$. Hence, as an immediate consequence of Theorem \ref{thm:crosspolytope_spectrum}, we obtain:

\lowerbound*

\section{The $n$ largest eigenvalues of the complete graph $K_n$}\label{sec:top-n-ev}

We proceed to establish a lower bound for the sum of the largest $n$ eigenvalues of $L(K_n,p)$ for all embeddings $p:[n]\to \Rea^d$. As a corollary, we  derive an upper bound for $a_d(K_n)$.

\begin{lemma}\label{lem:one_third}
Let $p:[n]\to\Rea^d$ be injective. Then,
\[
\sum_{j=(d-1)n+1}^{d n}\lambda_j(L(K_n,p)) \ge 
\frac {n^2}3+n\,.
\]
\end{lemma}

\upperbound*
\begin{proof}[Proof of Theorem \ref{cor:upperbound}]

Let $p:V\to\Rea^d$ be injective. We compute the trace of $L(K_n,p)$
in two ways. By Lemma \ref{lemma:down_laplacian}, all the diagonal entries of $L^{-}(K_n,p)$ equal $2$, since $p$ is injective. Thus, we have
\begin{equation}\label{eq:trace_direct}
\mathrm{Tr}(L(K_n,p))=\mathrm{Tr}(L^{-}(K_n,p)) = 2\binom n2=n^2-n\,.
\end{equation}

On the other hand, let $\lambda=\lambda_{\binom{d+1}{2}+1}(L(K_n,p))$. We deduce from Lemma \ref{lem:one_third} that
\begin{align}\nonumber
\mathrm{Tr}(L(K_n,p))=&~\sum_{j=1}^{dn}\lambda_j(L(K_n,p)) 
\\\nonumber =&
\sum_{j=\binom{d+1}{2}+1}^{(d-1)n}\lambda_j(L(K_n,p)) +\sum_{j=(d-1)n+1}^{dn}\lambda_j(L(K_n,p)) 
\\ \ge&~ \left((d-1)n-\binom{d+1}2\right)\lambda+\frac {n^2}3+n.
\label{eq:trace_eigs}
\end{align}
By combining \eqref{eq:trace_direct} and \eqref{eq:trace_eigs}, we derive that
\[
\lambda \le  \frac{2n^2/3-2n}{(d-1)n-\binom{d+1}2}.
 \]
Finally, we have
\begin{align*}
    & \frac{2n^2/3-2n}{(d-1)n-\binom{d+1}2}-\frac{2n}{3(d-1)}
    =\frac{n(d(d+1)-6(d-1))}{3(d-1)\left((d-1)n-\binom{d+1}{2}\right)} 
    \\
    &\leq \frac{n(d(d+1)-6(d-1))}{3n(d-1)^2}
    =\frac{d(d+1)-6(d-1)}{3(d-1)^2}
 \leq \frac{1}{3}.
\end{align*}
Therefore, we obtain
\[
\lambda \le \frac{2n}{3(d-1)}+\frac{1}{3}.
\]
Thus, by Lemma \ref{lemma:a_d_for_embeddings}, we obtain
\[
a_d(K_n) \le \frac{2n}{3(d-1)}+\frac{1}{3},
\]
as claimed.
\end{proof}

To prove Lemma \ref{lem:one_third} we will need the following theorem due to Ky Fan:

\begin{theorem}[Ky Fan {\cite[Thm. 1]{fan1949theorem}}]\label{thm:fan}
Let $A\in\Rea^{m\times m} $ be a symmetric matrix, and let $\mu_1\geq \mu_2\geq \cdots\geq \mu_m$ be its eigenvalues. Then, for every $k\leq m$,
\[
\sum_{i=1}^k \mu_i = \max\left\{ \text{Tr}(AP):\, P\in\mathcal{P}_k\right\}
\]
where $\mathcal{P}_k$ consists of all orthogonal projection matrices into $k$-dimensional subspaces of $\Rea^m$.
\end{theorem}
\begin{proof}[Proof of Lemma \ref{lem:one_third}]
Let $E=\binom{[n]}{2}$. For $i\in[n]$, let  $v^{(i)}\in \Rea^E$ be defined by
\[
v^{(i)}_e=\begin{cases}
    1 & \text{ if } i\in e,\\
    0 & \text{ otherwise.}
\end{cases}
\]
We claim that the 
$E\times E$ symmetric matrix $P$ defined by
\[
P_{e,e'} =\left\{\begin{matrix}
\frac{2}{n-1}&e=e'\\[6pt]
\frac{n-3}{(n-1)(n-2)}&|e\cap e'|=1\\[6pt]
\frac{-1}{\binom {n-1}2}&e\cap e'=\emptyset
\end{matrix}\right.\,,
\]
is the orthogonal projection on the subspace spanned by $\{v^{(i)} :\, i\in[n]\}$. Indeed, if $e=\{i,j\}$, then the $e$-th row of $P$ is equal to 
\[
P_{e,\cdot} = \frac{1}{n-1}\left(v^{(i)}+v^{(j)}-\frac{1}{n-2}\sum_{k\ne i,j}v^{(k)}\right)\,.
\]
Therefore, $Pw=0$ for every $w$ that is orthogonal to all the vectors $v^{(i)},i\in[n]$. In addition, a straightforward computation shows that $Pv^{(i)}=v^{(i)}$ for every $i\in[n]$ since  $(v^{(i)})^tv^{(j)}=1$ if $i\ne j$ and $\|v^{(i)}\|^2=n-1$.

We apply theorem \ref{thm:fan} for $L^-:=L^-(K_n,p)$
and $P$ to find that
\begin{align*}
\sum_{i=(d-1)n+1}^{dn}\lambda_{i}(L(K_n,p))=
\sum_{i=\binom{n-1}{2}}^{\binom{n}{2}}\lambda_i(L^-) \ge&~ \mathrm{Tr}( L^-P)=\sum_{e,e'}P_{e,e'}L^-_{e,e'}.
\end{align*}
Recall the precise description of $L^-$ in Lemma \ref{lemma:down_laplacian}. Since $p$ is injective, the contribution of the diagonal terms $e=e'$ to the summation is
\begin{equation}
\sum_{e\in E}P_{e,e}L^-_{e,e} = \binom{n}2\cdot \frac{2}{n-1}\cdot 2=2n\,.
\label{eq:diagonal}    
\end{equation}
In addition, since $L^-_{e,e'}=0$ if $e\cap e'=\emptyset$, the contribution of the non-diagonal terms is
\begin{equation}\label{eq:ee'}
\sum_{e\ne e'}P_{e,e'}L^-_{e,e'}=\sum_{i=1}^{n}\sum_{j\ne i}\sum_{j'\ne i,j} \frac{n-3}{(n-1)(n-2)}L^-_{\{i,j\},\{i,j'\}}\,.    
\end{equation}

Note that all the $\binom n3$ triples $\{i,j,j'\}$ of vertices satisfy
\[
L^-_{\{i,j\},\{i,j'\}}+L^-_{\{j,i\},\{j,j'\}}+L^-_{\{j',i\},\{j',j\}}\ge 1.
\]
Indeed, since $p$ is injective, this is the sum of the cosines of the angles in the (possibly flat) triangle $p(i),p(j),p(j')$--- which is bounded from below by $1$. In addition, note that each term $L^-_{\{i,j\},\{i,j'\}}$ appears twice in \eqref{eq:ee'} since $j,j'$ are ordered. Consequently,
\begin{equation}
\sum_{e\ne e'}P_{e,e'}L^-_{e,e'}\ge \binom{n}3 \cdot \frac{2(n-3)}{(n-1)(n-2)}=\frac{n^2}{3}-n\,.
\label{eq:off_diagonal}    
\end{equation}
Joining \eqref{eq:diagonal} and \eqref{eq:off_diagonal}, we obtain
\[
\sum_{i=(d-1)n+1}^{dn}\lambda_{i}(L(K_n,p))\geq 
\frac {n^2}3+n\,.
\]
\end{proof}

We believe that the bound in Lemma \ref{lem:one_third} can be improved:

\begin{conjecture}\label{conj:large_eigenvalues}
Let $p:[n]\to\Rea^d$ be injective. Then, the sum of the $n$ largest eigenvalues of $L(K_n,p)$ is at least $\frac{n(n+1)}{2}$.
\end{conjecture}

Recall that, by Lemma \ref{lemma:n_eigenvalue}, for every non-constant $p:[n]\to \Rea^d$, the largest eigenvalue of $L(K_n,p)$ is $n$. Thus, Conjecture \ref{conj:large_eigenvalues} is equivalent to saying that the average of the next $n-1$ largest eigenvalues is at least $n/2$.

\section{Concluding remarks}\label{sec:conclude}

\subsection{Complete graphs}

We conjecture that the lower bound of Theorem \ref{thm:lower_bound} is essentially tight: 
\begin{conjecture}
Let $n\geq 2d$. Then, 
\[
\lfrac{n}{2d}\leq a_d(K_n)\leq \frac{n}{2d}.
\]
\end{conjecture}
 
\subsection{Regular graphs}
From some computer calculations, it seems possible that the following generalization of \cite[Conj. 2]{jordan2020rigidity} holds:
\begin{conjecture}
For $d\geq 1$,
\[
    \lim_{n\to\infty} \max\{ a_d(G) :\, G \text{ is $2d$-regular on $n$ vertices}\}=0.
\]
\end{conjecture}

\subsection{Repeated points}
In line with our analysis of the minimal nontrivial eigenvalue of $L(G,p)$, where $G$ is the Tur{\'a}n graph $T(n,d+1)$ (resp. $T(n,2d)$) and $p=q^{\triangle}$ (resp. $p=q^{\diamond}$) in Theorem~\ref{thm:simplex_repeated} and Theorem~\ref{thm:crosspolytope_spectrum}, we conjecture the following general phenomena regarding the effect of repeated points on the spectral gap.

For an injective $p: [n] \to \Rea^d$ and graph $G$ with vertex set $[n]$, denote by $\lambda(G,p)=\lambda_{\binom{d+1}{2}+1}(L(G,p))$.  
For $k\ge 1$ let $p^{k}: [kn] \to \Rea^d $ be a mapping such that $|(p^{k})^{-1}(p(v))|=k$ for every $v\in [n]$. In words, we put $k$ vertices on each point of the image of $p$.

\begin{conjecture}
For every injective mapping $p: [n] \to \Rea^d$ and every $k\ge 2$, 
\[\lambda(K_{kn},p^{k})= \frac{k}{2}\lambda(K_{2n},p^{2}).
\]
\end{conjecture}
We remark that for $k=1$ the assertion fails, as demonstrated by the regular simplex embedding $p^{\triangle}$.

\bibliographystyle{abbrv}
\bibliography{biblio}

\end{document}